\documentclass[12pt,reqno]{amsart}
\usepackage{amssymb}
\usepackage{epsfig}
\usepackage{mathdots}
\usepackage{hyperref}
\usepackage{fullpage}
\usepackage{color}

\usepackage{array, longtable}
\newcolumntype{C}{>{$}c<{$}}

\theoremstyle{plain}
\newtheorem{theorem}{Theorem}[section]
\newtheorem{corollary}{Corollary}[section]

\newtheorem{lemma}{Lemma}[section]
\newtheorem{proposition}{Proposition}[section]

\theoremstyle{definition}

\newtheorem{remark}{Remark}[section]

\numberwithin{equation}{section}

\setbox0=\hbox{$+$}
\newdimen\plusheight
\plusheight=\ht0
\def\+{\;\lower\plusheight\hbox{$+$}\;}

\setbox0=\hbox{$-$}
\newdimen\minusheight
\minusheight=\ht0
\def\-{\;\lower\minusheight\hbox{$-$}\;}

\setbox0=\hbox{$\cdots$}
\newdimen\cdotsheight
\cdotsheight=\plusheight
\def\cds{\lower\cdotsheight\hbox{$\cdots$}}

\begin{document}

\title{On the Petersson norm  $\langle\theta_{\psi},\theta_{\psi}\rangle$}

\author{Wei-Lun Tsai and Dongxi Ye}

\address{Department of Mathematics, University of South Carolina,
 1523 Greene St LeConte College Rm 450
 Columbia, SC 29208}

\email{weilun@mailbox.sc.edu}

\address{
Beijing Normal-Hong Kong Baptist University, Zhuhai 519082, Guangdong, 
People's Republic of China}

\email{dongxiye@bnbu.edu.cn}


\thanks{Wei-Lun Tsai was partially supported by the SEC Faculty Grant. Dongxi Ye was supported by the Guangdong Basic and Applied Basic Research Foundation (Grant No. 2024A1515030222) and the BNBU Start-up Research Fund (Grant No. R0700157-26). }

\begin{abstract}
Let $K$ be an imaginary quadratic field of discriminant~$-d<-11$, and for $\ell\geq1$, let $\psi$ be a Hecke character of $K$ with trivial finite conductor and infinite type~$(2\ell,0)$.  In this work we prove that for any prime $p>3$, the quotient 
$$
\frac{\langle\theta_{\psi},\theta_{\psi}\rangle}{\Omega_{K}^{4\ell}},
$$
is $\lambda$-integral for a prime ideal $\lambda$ over $p$, where $\langle\theta_{\psi},\theta_{\psi}\rangle$ is the Petersson norm of the theta function $\theta_{\psi}=\theta_{\psi}(\tau)$ attached to $\psi$, and $\Omega_{K}$ denotes the Chowla--Selberg period attached to~$K$, and the product 
$$
\prod_{i=1}^{h_{K}}\frac{\langle\theta_{\psi_{i}},\theta_{\psi_{i}}\rangle}{\Omega_{K}^{4\ell}}
$$
is rational, where the product is over all $h_{K}$ of the underlying Hecke characters.    
\end{abstract}

\maketitle

\allowdisplaybreaks


\section{Introduction}




Let $K$ be an imaginary quadratic field of discriminant~$-d<-11$, and denote by $\mathcal{O}_{K}$ the ring of integers of~$K$. For $\ell\geq 1$, let $\psi$ be a Hecke character of $\mathcal{O}_{K}$ with infinite type $(2\ell,0)$. Then one can associate with $\psi$ a theta function $\theta_{\psi}(\tau)$  defined for $\tau=x+iy$ in the upper half plane $\mathbb{H}$ by 
$$
\theta_{\psi}(\tau):=\sum_{\substack{\mathcal{A}\subset\mathcal{O}_{K}\\ideal}}\psi(\mathcal{A})q^{N(\mathcal{A})},
$$
where henceforth, $q:=e^{2\pi i\tau}$, and $N(\mathcal{A})$ denotes the norm of an ideal~$\mathcal{A}$.
Such a theta function is known to be a CM newform of level $\Gamma_{0}(d)$ and weight~$2\ell+1$ \cite{Hecke},  so it is natural to consider the Petersson norm of $\theta_{\psi}=\theta_{\psi}(\tau)$:
$$
\langle\theta_{\psi},\theta_{\psi}\rangle:=\int_{\Gamma_{0}(d)\backslash\mathbb{H}}|\theta_{\psi}|^{2}y^{2\ell+1}\frac{dxdy}{y^{2}}.
$$
This yields a transcendental number that is an algebraic scalar multiple of $\Omega_{K}^{4\ell}$  \cite{Shi, Shi2}, where
\begin{equation}
\label{omegak}
    \Omega_{K}:=\frac{1}{\sqrt{4\pi d}}\left(\prod_{n=1}^{d-1}\Gamma\left(\frac{n}{d}\right)^{\chi_{-d}(n)}\right)^{\frac{1}{2h_{K}}}
\end{equation}
is the Chowla--Selberg period \cite{CS} attached to~$K$ with $\chi_{-d}(\cdot)$ and $h_{K}$ respectively denoting the quadratic character associated with~$K$ and the class number of~$K$.
The first main result of the present work is the integrality property for that algebraic scalar.



\begin{theorem}\label{main}
    Follow the assumption and notation above.
    Then for any prime~$p>3$, the algebraic number
   \begin{align}
       \label{TTO}
       \frac{\langle\theta_{\psi},\theta_{\psi}\rangle }{\Omega_{K}^{4\ell}}
   \end{align}
     is $\lambda$-integral for any prime ideal~$\lambda$ over~$p$.
\end{theorem}

The second main result is  the rationality of the product  of~\eqref{TTO} over all $\psi$'s.

\begin{theorem}
\label{mainthm2}
    Follow the assumption and notation above. Then the product
    \begin{align}
        \label{productheta}
        \prod_{i=1}^{h_{K}}\frac{\langle\theta_{\psi_{i}},\theta_{\psi_{i}}\rangle }{\Omega_{K}^{4\ell}}
    \end{align}
    is a rational number.
\end{theorem}

Combining Theorems~\ref{main} and~\ref{mainthm2}, we immediately have the following.
\begin{corollary}
    The rational number~\eqref{productheta} is $p$-integral for any prime $p>3$.
\end{corollary}

The idea of proof of Theorem~\ref{main} is to express the quotient~\eqref{TTO} in terms of CM values of (almost holomorphic) modular functions for ${\rm SL}_{2}(\mathbb{Z})$ and upon this, analyze the integrality of these CM values by the membership of the modular functions in the integral closure of $\mathbb{Z}[\frac{1}{6}][j(\tau)]$, where $j(\tau)$ denotes the modular~$j$-invariant. 
The first step can be realized by an alternative formulation for $\Omega_{K}$ in terms of CM values of the Dedekind eta function, and a result due to Simard \cite{Si} expressing the Petersson norm of $\theta_{\psi}$ as a linear combination of CM values of quasi-modular forms. These are to be introduced at the beginning of Section~\ref{prelim} and will lead us to deliberately write the quotient~\eqref{TTO} as a product of CM values of certain particular (almost holomorphic) modular functions. Following these, we discuss the integrality of these algebraic numbers that will eventually lead to the proof of Theorem~\ref{main} given in Section~\ref{proofthm}. 

For Theorem~\ref{mainthm2}, we first express~\eqref{productheta} as the determinant of a matrix of CM values of some almost meromorphic modular function for~${\rm SL}_{2}(\mathbb{Z})$ using another result of Simard; see Subsection~\ref{alterform}.  This alternative formulation will be seen to neutralize the effect of the Hecke characters~$\psi$'s and be a CM average over the ideal class group of~$K$. Following this, we can apply Shimura's reciprocity law and class field theory to prove its rationality. The proof will be given in Section~\ref{proofthm2}.

\section{Nuts and bolts}
\label{prelim}

To discuss and present the preliminaries for the proofs of Theorems~\ref{main} and~\ref{mainthm2}, we first introduce and fix some notation.

As a convention, we write $\eta(\tau)$, $E_{4}(\tau)$ and $E_{6}(\tau)$ for the Dedekind eta function, and normalized Eisenstein series of weight~$4$ and~$6$ for ${\rm SL}_{2}(\mathbb{Z})$, respectively. However, to align with the notation used in a result by Simard, which will play a crucial role in the proof of the theorem and is to be introduced in Subsection~\ref{quotientcm}, we adopt $E_{2}(\tau)$ for the non-normalized almost holomorphic Eisenstein series of weight~$2$ for ${\rm SL}_{2}(\mathbb{Z})$, i.e.,
$$
E_{2}(\tau)=-\frac{1}{24}+\sum_{n=1}^{\infty}\left(\sum_{d|n}d\right)q^{n}+\frac{1}{8\pi{\rm Im}(\tau)}.
$$

Next, we write ${\rm Cl}(K)$ for the class group of~$K$, and for a nonzero integral ideal $\mathcal{A}=[a,\frac{b+\sqrt{-d}}{2}]$ of $K$, we write $\tau_{\mathcal{A}}=\frac{b+\sqrt{-d}}{2a}$ for the CM point in $\mathbb{H}$ attached to~$\mathcal{A}$. So it is well known by the CM theory of elliptic curves that the CM values $\eta(\tau_{\mathcal{A}})^{2}$ and $E_{m}(\tau_{\mathcal{A}})$ are all algebraic scalar multiples of some power of $\Omega_{K}$.

Finally, for $k\geq 0$, we define the Maass--Shimura operator $\delta_{k}$ of weight~$k$ by
$$
\delta_{k}(f):=\frac{1}{2\pi i}\left(\frac{\partial f}{\partial\tau}+\frac{k}{2{\rm Im}(\tau)i}f\right)
$$
for a complex function $f=f(\tau)$ in $\mathbb{H}$. So it is straightforward to check by definition that
 \begin{equation}
 \label{deltast}
         \delta_{s+t}(fg)=\delta_{s}(f)g+f\delta_{t}(g),
 \end{equation}
and for $f$ of weight-$k$ modular, $\delta_{k}(f)$ is of weight-$(k+2)$ modular. In particular, for $n\geq 1$, we write $$\delta_{k}^{n}:=\delta_{k+2(n-1)}\circ\delta_{k+2(n-2)}\circ\cdots\circ \delta_{k}.$$


\subsection{The quotient~\eqref{TTO} as a CM value}\label{quotientcm}

We aim to express the quotient~\eqref{TTO} in terms of CM values of (almost holomorphic) modular functions. To this end, what we need are CM-value formulations for $\Omega_{K}$ and $\langle\theta_{\psi},\theta_{\psi}\rangle$. 

\begin{lemma}\label{newperiod}
Let $\Omega_{K}$ be the Chowla--Selberg period attached to~$K$ defined by~\eqref{omegak}. Then we have
\begin{equation}
    \label{omegaketa}
    \Omega_{K}=\frac{1}{\sqrt{2}}\left(\prod_{[\mathcal{A}]\in {\rm Cl}(K)}N(\mathcal{A})^{-\frac{1}{2}}\left|\eta(\tau_{\mathcal{A}})\right|^{2}\right)^{\frac{1}{h_{K}}}.
\end{equation}

\end{lemma}

\begin{proof}
    { This is classical. See, e.g., \cite{PW}.}
\end{proof}

The following formula expressing the Petersson norm $\langle\theta_{\psi},\theta_{\psi}\rangle$ as a linear combination of CM values of $\delta_{2}^{2\ell-1}(E_{2})$ is due to Simard \cite[Theorem 15]{Si}.

\begin{lemma}\label{newthetaE2}
    Follow the definitions before. Then we have
    $$
 {\langle \theta_{\psi},\theta_{\psi}\rangle }=(d/4)^{\ell}h_{K}\sum_{[\mathcal{A}]\in {\rm Cl}(K)}\frac{\psi(\mathcal{A})^{2}}{N(\mathcal{A})^{4\ell}}(\delta_{2}^{2\ell-1}(E_{2}))(\tau_{\mathcal{A}}).
    $$
\end{lemma}

\begin{remark}
    There was a factor of $4/w_{K}^{2}$ in Simard's original formula, where $w_{K}$ denotes the number of units of~$K$. It is just~$1$ in our situation where the discriminant of~$K$ is assumed to be smaller than~$-11$. In addition, the absence of the factor of $N(\mathcal{A})^{4\ell}$ in the original formula is due to that in Simard's work, $E_{2}$ was viewed as a function in lattice.
\end{remark}


Clearly, making use of these two lemmas, we can express the quotient~\eqref{TTO} in terms of CM values of (almost holomorphic) modular functions. On top of this, we can further deliberately manipulate to write it as a product of CM values grouped by their integrality nature that one shall see in later subsections. Since integral exponentiation makes no difference to the integrality of an algebraic number, we instead consider the $6h_{K}$-th power of the quotient~\eqref{TTO} for convenience.

\begin{proposition}
    \label{TTOCM}
    Follow the definitions before. Then we have    
    \small
  \begin{align}\label{CMFORM}
     &\left(\frac{\langle \theta_{\psi},\theta_{\psi}\rangle }{\Omega_{K}^{4\ell}}\right)^{6h_{K}}=\left(h_{K}^{6h_{K}}\prod_{[\mathcal{A}]\in {\rm Cl}(K)}N(\mathcal{A})^{24\ell}\right)\left(\sum_{[\mathcal{A}]\in {\rm Cl}(K)}\frac{\psi(\mathcal{A})^{2}}{N(\mathcal{A})^{4\ell}}\cdot\frac{ d^{\ell }(\delta_{2}^{2\ell-1}(E_{2}))(\tau_{\mathcal{A}})}{\eta(\tau_{\mathcal{A}})^{8\ell}}\cdot  \frac{\eta(\tau_{\mathcal{A}})^{8\ell}}{\eta(\tau_{\mathcal{O}_{K}})^{8\ell}}\right)^{6h_{K}}
     \\&\hspace{9.5cm}\times\prod_{[\mathcal{A}]\in {\rm Cl}(K)}\left(\frac{1}{N(\mathcal{A})^{12}}\left|\frac{\eta(\tau_{\mathcal{O}_{K}})^{24}}{\eta(\tau_{\mathcal{A}})^{24}}\right|^{2}\right)^{\ell}.\notag
\end{align}
\normalsize
\end{proposition}

\begin{proof}
Starting by Lemmas~\ref{newperiod} and~\ref{newthetaE2}, one can deduce line by line that
\small
    \begin{align*}
     \left(\frac{\langle \theta_{\psi},\theta_{\psi}\rangle }{\Omega_{K}^{4\ell}}\right)^{6h_{K}}
     &=\frac{(d/4)^{6\ell h_{K}}h_{K}^{6h_{K}}\left(\sum_{[\mathcal{A}]\in {\rm Cl}(K)}\psi(\mathcal{A})^{2}N(\mathcal{A})^{-4\ell}(\delta_{2}^{2\ell-1}(E_{2}))(\tau_{\mathcal{A}})\right)^{6h_{K}}}{(1/4)^{6\ell h_{K}}\left(\prod_{[\mathcal{A}]\in {\rm Cl}(K)}N(\mathcal{A})^{-\frac{1}{2}}\left|\eta(\tau_{\mathcal{A}})\right|^{2}\right)^{24\ell}}\\
    &=\frac{d^{6\ell h_{K}}h_{K}^{6h_{K}}\left(\sum_{[\mathcal{A}]\in {\rm Cl}(K)}\psi(\mathcal{A})^{2}N(\mathcal{A})^{-4\ell}(\delta_{2}^{2\ell-1}(E_{2}))(\tau_{\mathcal{A}})\right)^{6h_{K}}}{\left(\prod_{[\mathcal{A}]\in {\rm Cl}(K)}N(\mathcal{A})^{-\frac{1}{2}}\left|\eta(\tau_{\mathcal{A}})\right|^{2}\right)^{24\ell}}\\
     &=\left(\prod_{[\mathcal{A}]\in {\rm Cl}(K)}N(\mathcal{A})^{12\ell}\right)\frac{d^{6\ell h_{K}}h_{K}^{6h_{K}}\left(\sum_{[\mathcal{A}]\in {\rm Cl}(K)}\psi(\mathcal{A})^{2}N(\mathcal{A})^{-4\ell}(\delta_{2}^{2\ell-1}(E_{2}))(\tau_{\mathcal{A}})\right)^{6h_{K}}}{\eta(\tau_{\mathcal{O}_{K}})^{48\ell h_{K}}}\\
     &\hspace{9cm}\times\frac{\eta(\tau_{\mathcal{O}_{K}})^{48\ell h_{K}}}{\left(\prod_{[\mathcal{A}]\in {\rm Cl}(K)}\left|\eta(\tau_{\mathcal{A}})\right|^{24}\right)^{2\ell}}\\
     &=\left(\prod_{[\mathcal{A}]\in {\rm Cl}(K)}N(\mathcal{A})^{12\ell}\right)d^{6\ell h_{K}}h_{K}^{6h_{K}}\left(\sum_{[\mathcal{A}]\in {\rm Cl}(K)}\frac{\psi(\mathcal{A})^{2}N(\mathcal{A})^{-4\ell}(\delta_{2}^{2\ell-1}(E_{2}))(\tau_{\mathcal{A}})}{\eta(\tau_{\mathcal{O}_{K}})^{8\ell}}\right)^{6h_{K}}\\
     &\hspace{9cm}\times\prod_{[\mathcal{A}]\in {\rm Cl}(K)}\left(\frac{\eta(\tau_{\mathcal{O}_{K}})^{24}}{\left|\eta(\tau_{\mathcal{A}})\right|^{24}}\right)^{2\ell}\\
     &=\left(h_{K}^{6h_{K}}\prod_{[\mathcal{A}]\in {\rm Cl}(K)}N(\mathcal{A})^{24\ell}\right)\left(\sum_{[\mathcal{A}]\in {\rm Cl}(K)}\frac{\psi(\mathcal{A})^{2}}{N(\mathcal{A})^{4\ell}}\cdot\frac{ d^{\ell }(\delta_{2}^{2\ell-1}(E_{2}))(\tau_{\mathcal{A}})}{\eta(\tau_{\mathcal{A}})^{8\ell}}\cdot  \frac{\eta(\tau_{\mathcal{A}})^{8\ell}}{\eta(\tau_{\mathcal{O}_{K}})^{8\ell}}\right)^{6h_{K}}\\
     &\hspace{7.25cm}\times\prod_{[\mathcal{A}]\in {\rm Cl}(K)}\left(\frac{1}{N(\mathcal{A})^{12}}\left|\frac{\eta(\tau_{\mathcal{O}_{K}})^{24}}{\eta(\tau_{\mathcal{A}})^{24}}\right|^{2}\right)^{\ell}
\end{align*}
\normalsize
as desired. Note that in the second to the last equality, the term $d^{\ell h_{K}}$ is incorporated into the parentheses in the middle to result in the last equality, and the last equality follows from the fact that $\eta(\tau_{\mathcal{O}_{K}})^{24}$ is a real number.
\end{proof}




\subsection{Integrality of CM values of eta quotients}

As mentioned earlier, the components in the CM formulation for the $6h_{K}$-th power of the quotient~\eqref{TTO} given in Proposition~\ref{TTOCM} are manipulated, so that the integrality of the entire subject can be analyzed component-wise. We shall first see that the CM values of the eta quotients involved are all algebraic integers. 

The integrality of the CM values of the eta quotients in~\eqref{CMFORM} is well known. See, e.g., \cite[Theorem 2, p.\,163--166]{L}.
\begin{lemma}\label{etaAetaO}
 Follow the definitions before.  The CM value
    $$
    \frac{\eta(\tau_{\mathcal{A}})^{8}}{\eta(\tau_{\mathcal{O}_{K}})^{8}}
    $$
    is an algebraic integer, and the CM value
    $$
    \frac{1}{N(\mathcal{A})^{12}}\left|\frac{\eta(\tau_{\mathcal{O}_{K}})^{24}}{\eta(\tau_{\mathcal{A}})^{24}}\right|^{2}
    $$
    is a unit.
\end{lemma}

\subsection{The almost holomorphic modular form $\delta_{2}^{2\ell-1}(E_{2})$}

Making use of the following lemma, one shall see in Proposition~\ref{deltaE2} that the almost holomorphic modular form $\delta_{2}^{2\ell-1}(E_{2})$ lies in the (graded) algebra generated by $E_{2}$, $E_{4}$ and $E_{6}$ over $\mathbb{Z}[\frac{1}{6}]$, which will play a significant role later in our analysis of the integrality of the CM values associated with $\delta_{2}^{2\ell-1}(E_{2})$ inside the middle parentheses on the right hand side of~\eqref{CMFORM}.

\begin{lemma}\label{DE} Follow the definitions before. Then  we have
    \begin{align*}
        \delta_{2}(E_{2})=\frac{1}{288}E_{4}-2E_{2}^{2},\quad
        \delta_{4}(E_{4})=-\frac{1}{3}E_{6}-8E_{2}E_{4},\quad
        \delta_{6}(E_{6})=-\frac{1}{2}E_{4}^{2}-12E_{2}E_{6}.
    \end{align*}
    
\end{lemma}

\begin{proof}
   We first recall
\[
E_2(\tau)
=
-\frac1{24}
+
\sum_{n\geq1}\sigma_1(n)q^n
+
\frac{1}{8\pi y},
\]
and
\[
E_4(\tau)=1+240\sum_{n\geq1}\sigma_3(n)q^n,
\qquad
E_6(\tau)=1-504\sum_{n\geq1}\sigma_5(n)q^n.
\]
The holomorphic part of \(E_2\), namely
\[
E_2-\frac{1}{8\pi y},
\]
satisfies Ramanujan's differential identities (see e.g. \cite[Section 2.3]{Ono})
\begin{align*}
\frac{1}{2\pi i}\frac{\partial}{\partial \tau}
\left(E_2-\frac{1}{8\pi y}\right)
=
\frac{1}{288}E_4
-
2\left(E_2-\frac{1}{8\pi y}\right)^2,\qquad
\frac{1}{2\pi i}\frac{\partial E_4}{\partial \tau}
=
-\frac13E_6
-
8\left(E_2-\frac{1}{8\pi y}\right)E_4,
\end{align*}

\vspace{0.1in}
\noindent
and
\begin{align*}
\frac{1}{2\pi i}\frac{\partial E_6}{\partial \tau}
=
-\frac12E_4^2
-
12\left(E_2-\frac{1}{8\pi y}\right)E_6.
\end{align*}

\vspace{0.1in}
\noindent
Moreover, combining this with
\[
\frac{1}{2\pi i}
\frac{\partial}{\partial \tau}
\left(\frac{1}{8\pi y}\right)
=
\frac{1}{32\pi^2y^2},
\]
we obtain
\begin{align*}
\delta_2(E_2)
&=
\frac{1}{2\pi i}\frac{\partial E_2}{\partial \tau}
-
\frac{1}{2\pi y}E_2  \\
&=
\frac{1}{288}E_4
-
2\left(E_2-\frac{1}{8\pi y}\right)^2
+
\frac{1}{32\pi^2y^2}
-
\frac{1}{2\pi y}E_2=\frac{1}{288}E_4
-
2E_2^2.
\end{align*}

\vspace{0.1in}
\noindent
Finally, an analogous calculation gives
\begin{align*}
\delta_4(E_4)
=
\frac{1}{2\pi i}\frac{\partial E_4}{\partial \tau}
-
\frac{1}{\pi y}E_4  
=
-\frac13E_6
-
8\left(E_2-\frac{1}{8\pi y}\right)E_4
-
\frac{1}{\pi y}E_4  
=
-\frac13E_6
-
8E_2E_4,
\end{align*}
and
\begin{align*}
\delta_6(E_6)
=
\frac{1}{2\pi i}\frac{\partial E_6}{\partial \tau}
-
\frac{3}{2\pi y}E_6  
=
-\frac12E_4^2
-
12\left(E_2-\frac{1}{8\pi y}\right)E_6
-
\frac{3}{2\pi y}E_6  
=
-\frac12E_4^2
-
12E_2E_6
.
\end{align*}
\end{proof}

\begin{proposition}\label{deltaE2}
For $\ell\geq1$, the following holds:
    $$
   \delta_{2}^{2\ell-1}(E_{2})=\sum_{\substack{2i+4j+6k=4\ell\\i,j,k\geq0}}c_{i,j,k}E_{2}^{i}E_{4}^{j}E_{6}^{k}
    $$
    for some $c_{i,j,k}\in\mathbb{Z}[\frac{1}{6}]$.
\end{proposition}

\begin{proof}
The result follows by (\ref{deltast}) and  Lemma \ref{DE}, and simplifying the coefficients.
\end{proof}

\subsection{Integrality of CM values associated with Eisenstein series}

Proposition~\ref{deltaE2} at the end of the preceding subsection enables us to express
\begin{equation}
    \label{deltaeee}
    \frac{(\delta_{2}^{2\ell-1}(E_{2}))(\tau_{\mathcal{A}})}{\eta(\tau_{\mathcal{A}})^{8\ell}}=\sum_{\substack{2i+4j+6k=4\ell\\i,j,k\geq0}}c_{i,j,k}\left(\frac{E_{2}(\tau_{\mathcal{A}})}{\eta(\tau_{\mathcal{A}})^{4}}\right)^{i}\left(\frac{E_{4}(\tau_{\mathcal{A})}}{\eta(\tau_{\mathcal{A}})^{8}}\right)^{j}\left(\frac{E_{6}(\tau_{\mathcal{A}})}{\eta(\tau_{\mathcal{A}})^{12}}\right)^{k}
\end{equation}
with $c_{i,j,k}\in\mathbb{Z}[\frac{1}{6}].$
This further boils our analysis about the integrality property of the right hand side of~\eqref{CMFORM} down to the cases of the quotients on the right hand side of~\eqref{deltaeee}.

The integrality of the ones associated with $E_{4}$ and $E_{6}$ are classical.
\begin{lemma}\label{E4E6}
    The CM values
    $$
    \frac{E_{4}(\tau_{\mathcal{A}})}{\eta(\tau_{\mathcal{A}})^{8}}\quad\mbox{and}\quad \frac{E_{6}(\tau_{\mathcal{A}})}{\eta(\tau_{\mathcal{A}})^{12}}
    $$
    are algebraic integers.
\end{lemma}

\begin{proof}
    This follows from noting that both
    $$
    \left(\frac{E_{4}(\tau)}{\eta(\tau)^{8}}\right)^{3}\quad\mbox{and}\quad  \left(\frac{E_{6}(\tau)}{\eta(\tau)^{12}}\right)^{2}
    $$
    are modular functions for ${\rm SL}_{2}(\mathbb{Z})$ with pole supported at the cusp $[i\infty]$ only and integer coefficients and thus, lie in $\mathbb{Z}[j(\tau)]$, and the classical fact that $j(\tau_{\mathcal{A}})$ is an algebraic integer.
\end{proof}

For the weight~$2$ Eisenstein series $E_{2}$, we have the following observations.

\begin{lemma}\label{fE2}
Let $\mathcal{A}=[a,\frac{-b+\sqrt{-d}}{2}]$ with $b^{2}-4ac=-d$ be an integral ideal of $K$.  Let \[
f_{\mathcal A} :=E_2 - E_2|_2 M_{\mathcal A},
\]
where
\[
M_{\mathcal A} =
\begin{pmatrix}
 -b & -c \\
 a & 0
\end{pmatrix},
\]
and $|_2$ denotes the usual slash operator of weight $2.$ Then \(f_{\mathcal A}(\tau)\) is a holomorphic modular form of weight $2.$
Moreover, if
\[
\tau_{\mathcal A} = \frac{-b + \sqrt{-d}}{2a},
\]
then
\[
E_2(\tau_{\mathcal A})
= \frac{-b + \sqrt{-d}}{2\sqrt{-d}}\,
f_{\mathcal A}(\tau_{\mathcal A}).
\]
\end{lemma}

\begin{proof}
    The proof is straightforward and follows from the modular transformation of $E_{2}$ and the fact that $\tau_{\mathcal{A}}=M_{\mathcal{A}}\cdot\tau_{\mathcal{A}}$. We omit the details.
\end{proof}



\begin{proposition}
    \label{fAetaA}
Follow the assumptions as in Lemma~\ref{fE2}. With $-d<-11$, for any prime $3<p\nmid ac$,  the CM value
    $$
    \frac{f_{\mathcal{A}}(\tau_{\mathcal{A}})}{\eta(\tau_{\mathcal{A}})^{4}}
    $$
    is $\lambda$-integral for any prime ideal~$\lambda$ over $p$.
\end{proposition}

\begin{proof}
    Recall by Lemma~\ref{fE2} that
    $$
    f_{\mathcal{A}}(\tau)=E_{2}-E_{2}|_{2}M_{\mathcal{A}},
    $$
    where
    $$
    M_{\mathcal{A}}=\begin{pmatrix}
        -b&-c\\a&0
    \end{pmatrix}
    $$
    provided that $\mathcal{A}=[a,\frac{-b+\sqrt{-d}}{2}]$ with $b^{2}-4ac=-d$. In addition, note that there is a $\gamma\in {\rm SL}_{2}(\mathbb{Z})$ such that
    $$
    \gamma M_{\mathcal{A}}=\begin{pmatrix}
        a/s&*\\0&sc
    \end{pmatrix}=M
    $$
    for some $s\mid a$, so one can check that
    $$
    acf_{\mathcal{A}}(\tau)=ac(E_{2}-E_{2}|_{2}M)
    $$
  is a holomorphic modular form of weight~$2$ for some congruence subgroup  with coefficients in $\mathbb{Z}[\frac{1}{6}]$.  Therefore, the quotient $\left(\frac{acf_{\mathcal{A}}(\tau)}{\eta(\tau)^{4}}\right)^{6}$ is a modular function for some congruence subgroup with pole supported at cusps only and coefficients in $\mathbb{Z}[\frac{1}{6}]$ and thus, must be integral over $\mathbb{Z}[\frac{1}{6}][j(\tau)]$. Since $j(\tau_{\mathcal{A}})$ is an algebraic integer, then for any prime $3<p\nmid ac$, the CM value $\frac{f_{\mathcal{A}}(\tau)}{\eta(\tau)^{4}}$ is $\lambda$-integral for a prime ideal $\lambda$ over $p$.
\end{proof}
The materials up to this point suffice for the proof of Theorem~\ref{main}. Readers who wish to see the proof may skip directly to Section~\ref{proofthm}.


\subsection{Alternative formulation for~\eqref{productheta}}
\label{alterform}
To prove Theorem~\ref{mainthm2}, we need an auxiliary theta function $\theta_{\mathcal{A},\ell}=\theta_{\mathcal{A},\ell}(\tau)$ defined by
\begin{equation}
    \label{thetaA}
\theta_{\mathcal{A},\ell}(\tau):=\sum_{\alpha\in\mathcal{A}}\alpha^{2\ell}q^{N(\alpha)/N(\mathcal{A})}
\end{equation}
for an ideal~$\mathcal{A}$, which is a cusp form of weight~$2\ell+1$ and level~$\Gamma_{0}(d)$ with character~$\left(\frac{-d}{\cdot}\right)$. Similar to the CM formulation for~$\langle\theta_{\psi},\theta_{\psi}\rangle$ given by Lemma~\ref{newthetaE2}, Simard \cite[Proposition 13]{Si} also expressed the Petersson norm $\langle\theta_{\mathcal{A}_{i},\ell},\theta_{\mathcal{A}_{j},\ell}\rangle$ in terms of CM values of $\delta_{2}^{2\ell-1}(E_{2})$.

\begin{lemma}
    \label{thetaE2}
      Let $\theta_{\mathcal{A},\ell}$ be defined by~\eqref{thetaA}. Then we have
    \begin{equation}
        \label{theteE2formula}
          \langle\theta_{\mathcal{A}_{i},\ell},\theta_{\mathcal{A}_{j},\ell}\rangle
=4(d/4)^{\ell}\sum_{\substack{[\mathcal{C}]\in {\rm Cl}(K)\\\mathcal{A}_{i}\mathcal{A}_{j}^{-1}\mathcal{C}^{2}=\lambda_{\mathcal{C}}\mathcal{O}_{K}}}\frac{\lambda_{\mathcal{C}}^{2\ell}}{N(\mathcal{C})^{4\ell}}(\delta_{2}^{2\ell-1}(E_{2}))(\tau_{\mathcal{C}}).
    \end{equation}
    In particular, if $[\mathcal{A}_{i}\mathcal{A}_{j}^{-1}]$ is not a principal genus class, then  $\langle\theta_{\mathcal{A}_{i},\ell},\theta_{\mathcal{A}_{j},\ell}\rangle=0$.
\end{lemma}

Moreover, Simard also found an expression of~\eqref{productheta} in terms of $\langle\theta_{\mathcal{A}_{i},\ell},\theta_{\mathcal{A}_{j},\ell}\rangle$ that leads to the following alternative formulation for~\eqref{productheta}.

\begin{lemma}
    \label{BA}
    Let $K$ be an imaginary quadratic field of discriminant $-d<-11$ and fix ideal classes $\{[\mathcal{A}_{1}],\ldots,[\mathcal{A}_{h_{K}}]\}$. Then
     \begin{align}
        \label{productheta2}
        \prod_{i=1}^{h_{K}}\frac{\langle\theta_{\psi_{i}},\theta_{\psi_{i}}\rangle }{\Omega_{K}^{4\ell}}&=\prod_{i=1}^{h_{K}}\frac{h_{K}}{2^{2\ell }N(\mathcal{A}_{i})^{4\ell}}\cdot{\det\left(\frac{\langle\theta_{\mathcal{A}_{i},\ell},\theta_{\mathcal{A}_{j},\ell}\rangle}{E_{4}(\tau_{\mathcal{A}_{j}})^{\ell}}\right)_{1\leq i,j\leq h_{K}}}\\
       &\hspace{4cm}\times \left(\prod_{i=1}^{h_{K}}\frac{E_{4}(\tau_{\mathcal{A}_{i}})^{\ell}}{\eta(\tau_{\mathcal{A}_{i}})^{8\ell}}\right)\left(\prod_{i=1}^{h_{K}}\frac{\eta(\tau_{\mathcal{A}_{i}})^{8\ell}}{\left|\eta(\tau_{\mathcal{A}_{i}})^{8\ell}\right|}\right).\nonumber
    \end{align}
\end{lemma}

\begin{proof}
    By \cite[Proposition~28]{Si}, we have
    \begin{align*}
\prod_{i=1}^{h_{K}}\frac{\langle\theta_{\psi_{i}},\theta_{\psi_{i}}\rangle }{\Omega_{K}^{4\ell}}&=\prod_{i=1}^{h_{K}}\frac{h_{K}}{N(\mathcal{A}_{i})^{2\ell}}\cdot\frac{\det\left({\langle\theta_{\mathcal{A}_{i},\ell},\theta_{\mathcal{A}_{j},\ell}\rangle}\right)_{1\leq i,j\leq h_{K}}}{\Omega_{K}^{4\ell h_{K}}}\\
        &=\prod_{i=1}^{h_{K}}\frac{h_{K}}{N(\mathcal{A}_{i})^{2\ell}}\cdot\frac{\det\left({\langle\theta_{\mathcal{A}_{i},\ell},\theta_{\mathcal{A}_{j},\ell}\rangle}\right)_{1\leq i,j\leq h_{K}}}{\prod_{i=1}^{h_{K}}E_{4}(\tau_{\mathcal{A}_{i}})^{\ell}}\cdot\frac{\prod_{i=1}^{h_{K}}E_{4}(\tau_{\mathcal{A}_{i}})^{\ell}}{\Omega_{K}^{4\ell h_{K}}}.
    \end{align*}
    Substituting~\eqref{omegaketa} into the right hand side and rearranging the resulting expression give~\eqref{productheta2}.
\end{proof}
The quotients are deliberately made up, so that as one shall see, Shimura's reciprocity law and class field theory can be accordingly applied. In what follows, we will separately discuss the rationality of the determinant and the rightmost products in~\eqref{productheta2}.

\subsection{The determinant in~\eqref{productheta2}}

We first see that the determinant in~\eqref{productheta2} lies in~$K$.

\begin{lemma}
    \label{determinantinK}
    Follow the assumptions and notations above. We have
    $$
    \det\left(\frac{\langle\theta_{\mathcal{A}_{i},\ell},\theta_{\mathcal{A}_{j},\ell}\rangle}{E_{4}(\tau_{\mathcal{A}_{j}})^{\ell}}\right)_{1\leq i,j\leq h_{K}}\in K.
    $$
\end{lemma}

\begin{proof}
First of all, by Lemma~\ref{productheta2}, we have
$$
\frac{\langle\theta_{\mathcal{A}_{i},\ell},\theta_{\mathcal{A}_{j},\ell}\rangle}{E_{4}(\tau_{\mathcal{A}_{j}})^{\ell}}=4(d/4)^{\ell}\sum_{\substack{[\mathcal{C}]\in {\rm Cl}(K)\\\mathcal{A}_{i}\mathcal{A}_{j}^{-1}\mathcal{C}^{2}=\lambda_{\mathcal{C}}\mathcal{O}_{K}}}\lambda_{\mathcal{C}}^{2\ell}\frac{\delta_{2}^{2\ell-1}(E_{2}))(\tau_{\mathcal{C}})}{E_{4}(\tau_{\mathcal{C}})}
\cdot\frac{E_{4}(\tau_{\mathcal{C}})}{E_{4}(\tau_{\mathcal{O}_{K}})}\cdot
\frac{E_{4}(\tau_{\mathcal{O}_{K}})}{E_{4}(\tau_{\mathcal{A}_{j}})}.
$$
It is known (see, e.g., \cite{S75}) that all of the quotients on the right hand side of the equation lie in the Hilbert class field $H$ of $K$. Since the associated (almost) meromorphic modular functions are all of level~${\rm SL}_{2}(\mathbb{Z})$, then by Shimura's reciprocity law \cite{S75},  for any $\sigma_{\mathcal{B}}\in{\rm Gal}(H/K)$ attached to an ideal class $[\mathcal{B}]$ via the Artin reciprocity law, we deduce that
    \begin{align*}
        \left(\frac{\langle\theta_{\mathcal{A}_{i},\ell},\theta_{\mathcal{A}_{j},\ell}\rangle}{E_{4}(\tau_{\mathcal{A}_{j}})^{\ell}}\right)^{\sigma_{\mathcal{B}}}&=4(d/4)^{\ell}\sum_{\substack{[\mathcal{C}]\in {\rm Cl}(K)\\\mathcal{A}_{i}\mathcal{A}_{j}^{-1}\mathcal{C}^{2}=\lambda_{\mathcal{C}}\mathcal{O}_{K}}}\lambda_{\mathcal{C}}^{2\ell}\frac{\delta_{2}^{2\ell-1}(E_{2}))(\tau_{\mathcal{C}\mathcal{B}^{-1}})}{E_{4}(\tau_{\mathcal{C}\mathcal{B}^{-1}})}
\cdot\frac{E_{4}(\tau_{\mathcal{C}\mathcal{B}^{-1}})}{E_{4}(\tau_{\mathcal{B}^{-1}})}\cdot
\frac{E_{4}(\tau_{\mathcal{B}^{-1}})}{E_{4}(\tau_{\mathcal{A}_{j}\mathcal{B}^{-1}})}\\
&=\frac{\langle\theta_{\mathcal{A}_{i}\mathcal{B},\ell},\theta_{\mathcal{A}_{j}\mathcal{B}^{-1},\ell}\rangle}{E_{4}(\tau_{\mathcal{A}_{j}\mathcal{B}^{-1}})^{\ell}},
    \end{align*}
and therefore, by basic matrix computation, we have
  \begin{align*}
      \left(\det\left(\frac{\langle\theta_{\mathcal{A}_{i},\ell},\theta_{\mathcal{A}_{j},\ell}\rangle}{E_{4}(\tau_{\mathcal{A}_{j}})^{\ell}}\right)_{1\leq i,j\leq h_{K}}\right)^{\sigma_{\mathcal{B}}}&=\det\left(\frac{\langle\theta_{\mathcal{A}_{i}\mathcal{B},\ell},\theta_{\mathcal{A}_{j}\mathcal{B}^{-1},\ell}\rangle}{E_{4}(\tau_{\mathcal{A}_{j}\mathcal{B}^{-1}})^{\ell}}\right)_{1\leq i,j\leq h_{K}}\\
      &=\det\left(\frac{\langle\theta_{\mathcal{A}_{i},\ell},\theta_{\mathcal{A}_{j},\ell}\rangle}{E_{4}(\tau_{\mathcal{A}_{j}})^{\ell}}\right)_{1\leq i,j\leq h_{K}}.
  \end{align*}
    This justifies the claim.
\end{proof}

\subsection{The products in~\eqref{productheta2}}

Next, the rightmost products of~\eqref{productheta2} both lie in the maximal abelian extension $K_{ab}$ of $K$.

\begin{lemma}
    \label{rightmostproduct}
    Follow the assumptions and notation above and fix a set of representatives $\{[\mathcal{A}]\}$ such that $\overline{\mathcal{A}}$ is among the representatives whenever $[\overline{\mathcal{A}}]\ne [\mathcal{A}]$.  We have
    $$
    \prod_{i=1}^{h_{K}}\frac{E_{4}(\tau_{\mathcal{A}_{i}})^{\ell}}{\eta(\tau_{\mathcal{A}_{i}})^{8\ell}},\quad\prod_{i=1}^{h_{K}}\frac{\eta(\tau_{\mathcal{A}_{i}})^{8\ell}}{\left|\eta(\tau_{\mathcal{A}_{i}})^{8\ell}\right|} \in K_{ab}.
    $$
\end{lemma}

\begin{proof}
Note that the function $E_{4}(\tau)/\eta(\tau)^{8}$ is a modular function of level~$\Gamma(3)$. So, by the CM theory of elliptic curves, e.g., \cite{L}, its CM values all lie in $K_{ab}$.

For the second product, first note that
$$
\frac{\eta(\tau)^{8}}{\left|\eta(\tau)^{8}\right|}=\frac{\eta(\tau)^{4}}{\eta(-\overline{\tau})^{4}},\quad\mbox{so}\quad \frac{\eta(\tau_{\mathcal{A}})^{8}}{\left|\eta(\tau_{\mathcal{A}})^{8}\right|}=\frac{\eta(\tau_{\mathcal{A}})^{4}}{\eta(\tau_{\overline{\mathcal{A}}})^{4}}.
$$
This implies that
$$
\prod_{i=1}^{h_{K}}\frac{\eta(\tau_{\mathcal{A}_{i}})^{8}}{\left|\eta(\tau_{\mathcal{A}_{i}})^{8}\right|}=\prod_{[\mathcal{A}]=[\overline{\mathcal{A}}]}\frac{\eta(\tau_{\mathcal{A}})^{4}}{\eta(\tau_{\overline{\mathcal{A}}})^{4}}.
$$
Next, recall that 
$$
(\eta(\tau)^{4})^{6}=\eta(\tau)^{24}=\frac{1}{1728}\left(E_{4}(\tau)^{3}-E_{6}(\tau)^{2}\right).
$$
Then if $\overline{\mathcal{A}}=\alpha\mathcal{A}$ for some nonzero $\alpha\in K$,  we have
$$
\frac{1}{1728}\left(E_{4}(\tau_{\overline{\mathcal{A}}})^{3}-E_{6}(\tau_{\overline{\mathcal{A}}})^{2}\right)=\frac{\alpha^{-12}}{1728}\left(E_{4}(\tau_{\mathcal{A}})^{3}-E_{6}(\tau_{\mathcal{A}})^{2}\right),
$$
and thus, $\eta(\tau_{\overline{\mathcal{A}}})^{4}=(\pm\zeta_{3})^{j}\alpha^{-2}\eta(\tau_{\mathcal{A}})$ for some integer~$j$, where $\zeta_{3}=e^{2\pi i/3}$. Therefore,
$$
\prod_{[\mathcal{A}]=[\overline{\mathcal{A}}]}\frac{\eta(\tau_{\mathcal{A}})^{4}}{\eta(\tau_{\overline{\mathcal{A}}})^{4}}\in K(\zeta_{3})\subset K_{ab}.
$$
This finishes the proof.
\end{proof}

Combining Lemma~\ref{rightmostproduct} and Lemma~\ref{determinantinK} in the preceding subsection and~\eqref{productheta2}, we have that the product~\eqref{productheta} actually lies in~$K_{ab}$.

\begin{proposition}
    \label{inKab}
    Follow the assumptions and notation above. We have
    $$
    \prod_{i=1}^{h_{K}}\frac{\langle\theta_{\psi_{i}},\theta_{\psi_{i}}\rangle }{\Omega_{K}^{4\ell}}\in K_{ab}.
    $$
\end{proposition}

\subsection{The cube of~\eqref{productheta}}
We just show that the product~\eqref{productheta} belongs to~$K_{ab}$, which is actually a real number by its definition. So, if it can be proved to lie in~$K$, the product  must be rational. To this end, the following result regarding the the cube of the product  is critical.

\begin{proposition}
    \label{cuberational}
     Follow the assumptions and notation above. We have
       $$
    \left(\prod_{i=1}^{h_{K}}\frac{\langle\theta_{\psi_{i}},\theta_{\psi_{i}}\rangle }{\Omega_{K}^{4\ell}}\right)^{3}\in K.
    $$
\end{proposition}

\begin{proof}
  Fix a set of representatives $\{[\mathcal{A}]\}$ such that $\overline{\mathcal{A}}$ is among the representatives whenever $[\overline{\mathcal{A}}]\ne [\mathcal{A}]$.   By~\eqref{productheta2}, we have
    \begin{align*}
        \left(\prod_{i=1}^{h_{K}}\frac{\langle\theta_{\psi_{i}},\theta_{\psi_{i}}\rangle }{\Omega_{K}^{4\ell}}\right)^{3}&=\left(\prod_{i=1}^{h_{K}}\frac{h_{K}}{2^{2\ell }N(\mathcal{A}_{i})^{4\ell}}\right)^{3}\cdot{\det\left(\frac{\langle\theta_{\mathcal{A}_{i},\ell},\theta_{\mathcal{A}_{j},\ell}\rangle}{E_{4}(\tau_{\mathcal{A}_{j}})^{\ell}}\right)^{3}_{1\leq i,j\leq h_{K}}}\cdot \left(\prod_{i=1}^{h_{K}}\frac{E_{4}(\tau_{\mathcal{A}_{i}})^{3}}{\eta(\tau_{\mathcal{A}_{i}})^{24}}\right)^{\ell}\\
        &\hspace{9cm}\times \left(\prod_{i=1}^{h_{K}}\frac{\eta(\tau_{\mathcal{A}_{i}})^{24}}{\left|\eta(\tau_{\mathcal{A}_{i}})^{24}\right|}\right)^{\ell}.
    \end{align*}
    By Lemma~\ref{determinantinK}, the determinant is in~$K$. 
    Next, note that $E_{4}(\tau)^{3}/\eta(\tau)^{24}$ is just the modular $j$-invariant $j(\tau)$, and thus,
    $$
    \prod_{i=1}^{h_{K}}\frac{E_{4}(\tau_{\mathcal{A}_{i}})^{3}}{\eta(\tau_{\mathcal{A}_{i}})^{24}}=N_{\mathbb{Q}(j(\tau_{\mathcal{O}_{K}}))/\mathbb{Q}}(j(\tau_{\mathcal{O}_{K}}))\in\mathbb{Q}.
    $$
    For the last product, apply the argument used in the proof of Lemma~\ref{rightmostproduct} to show that
    $$
    \frac{\eta(\tau_{\mathcal{A}})^{24}}{\left|\eta(\tau_{\mathcal{A}})^{24}\right|}=\frac{\eta(\tau_{\mathcal{A}})^{12}}{\eta(\tau_{\overline{\mathcal{A}}})^{12}}=\pm\alpha^6
    $$
    for some nonzero $\alpha\in K$. Combining these yields the conclusion.
\end{proof}

Building upon Propositions~\ref{inKab} and~\ref{cuberational}, the proof of Theorem~\ref{mainthm2} will be presented in Section~\ref{proofthm2}.

\section{Proof of Theorem~\ref{main}}
\label{proofthm}



Note that as an average over ${\rm Cl}(K)$ by Lemmas~\ref{newperiod} and~\ref{newthetaE2}, the quotient~$\frac{\langle \theta_{\psi},\theta_{\psi}\rangle }{\Omega_{K}^{4\ell}}$ is independent of the choice of the representatives $\mathcal{A}=[a,\frac{-b+\sqrt{-d}}{2}]$ with $-d=b^{2}-4ac$. So, for a fixed prime~$p>3$, we want to select representatives $\mathcal{A}$'s such that $p\nmid ac$ universally, so that Proposition~\ref{fAetaA} can be applied to all the cases simultaneously. Such a selection is guaranteed by the following technical  lemma.

\begin{lemma}
    \label{techlem}
    Fix a prime~$p>3$. There is a set of integral representatives $\{\mathcal{A}_{i}\}$ with $\mathcal{A}_{i}=[a_{i},\frac{-b_{i}+\sqrt{-d}}{2}]$ and $-d=b_{i}^{2}-4a_{i}c_{i}$ for ${\rm Cl}(K)$ satisfying that $p\nmid a_{i}c_{i}$ for $i=1,\ldots,h_{K}$.
\end{lemma}

\begin{proof}
    Viewing $\mathcal{A}_{i}$ as a primitive quadratic form $a_{i}X^{2}+b_{i}XY+c_{i}Y^{2}$, it is equivalent to proving that there is a $\begin{pmatrix}
        A&B\\ C&D
    \end{pmatrix}\in{\rm SL}_{2}(\mathbb{Z})$ such that 
    $$
    p\nmid (a_{i}A^{2}+b_{i}AC+c_{i}C^{2})(a_{i}B^{2}+b_{i}BD+c_{i}D^{2}).
    $$
    Such a transformation can be guaranteed by the fact that the reduction map ${\rm SL}_{2}(\mathbb{Z})\to {\rm SL}_{2}(\mathbb{Z}/p\mathbb{Z})$ is surjective. If $p\nmid a_{i}c_{i}$, we are done. Otherwise, we can assume $p\mid a_{i}$ by applying the transformation $ \begin{pmatrix}
        0&-1\\1&0
    \end{pmatrix}$.
Upon this,  if $p\mid a_i$, $p\mid b_{i}$ and $p\nmid c_{i}$, pick
        $$
        \begin{pmatrix}
        A&B\\ C&D
    \end{pmatrix}\equiv \begin{pmatrix}
        1&0\\ 1&1
    \end{pmatrix}\pmod{p}.
        $$
If $p\mid a_{i}$ and $p\nmid b_{i}c_{i}$, pick
$$
 \begin{pmatrix}
        A&B\\ C&D
    \end{pmatrix}\equiv \begin{pmatrix}
        b_{i}^{-1}(1-c_{i})&b_{i}^{-1}(1-c_{i}(1+b_{i}))\\ 1&1+b_{i}
    \end{pmatrix}\pmod{p}.
$$
If $p\mid a_{i}$, $p\mid c_{i}$ and $p\nmid b_{i}$, pick
$$
 \begin{pmatrix}
        A&B\\ C&D
    \end{pmatrix}\equiv \begin{pmatrix}
        1&1\\ 1&2
    \end{pmatrix}\pmod{p}.
$$
\end{proof}

Now we are ready to state the proof of Theorem~\ref{main}.

\begin{proof}[Proof of Theorem~\ref{main}]
Fix a prime~$p>3$ and upon this, fix a set of integral representatives $\{\mathcal{A}\}$ for ${\rm Cl}(K)$ as given by Lemma~\ref{techlem}. 

First note that $\psi(\mathcal{A})$ is an algebraic integer, since $\mathcal{A}^{h_{K}}=\beta\mathcal{O}_{K}$ for some algebraic integer~$\beta$, and thus, $\psi(\mathcal{A})^{h_{K}}=\psi(\mathcal{A}^{h_{K}})=\psi(\beta\mathcal{O}_{K})=\beta^{2\ell}$ an algebraic integer, whence $\frac{\psi(\mathcal{A})^{2}}{N(\mathcal{A})^{4\ell}}$ is an algebraic integer, since $N(\mathcal{A})$ is a $p$-adic unit, i.e., $p\nmid a=N(\mathcal{A})$, by Lemma~\ref{techlem}.

Recall by~\eqref{CMFORM} in Proposition~\ref{TTOCM} that
  \begin{align*}
    &\left(\frac{\langle \theta_{\psi},\theta_{\psi}\rangle }{\Omega_{K}^{4\ell}}\right)^{6h_{K}}=\left(h_{K}^{6h_{K}}\prod_{[\mathcal{A}]\in {\rm Cl}(K)}N(\mathcal{A})^{24\ell}\right)\left(\sum_{[\mathcal{A}]\in {\rm Cl}(K)}\frac{\psi(\mathcal{A})^{2}}{N(\mathcal{A})^{4\ell}}\cdot\frac{ d^{\ell }(\delta_{2}^{2\ell-1}(E_{2}))(\tau_{\mathcal{A}})}{\eta(\tau_{\mathcal{A}})^{8\ell}}\cdot  \frac{\eta(\tau_{\mathcal{A}})^{8\ell}}{\eta(\tau_{\mathcal{O}_{K}})^{8\ell}}\right)^{6h_{K}}
     \\&\hspace{9.5cm}\times\prod_{[\mathcal{A}]\in {\rm Cl}(K)}\left(\frac{1}{N(\mathcal{A})^{12}}\left|\frac{\eta(\tau_{\mathcal{O}_{K}})^{24}}{\eta(\tau_{\mathcal{A}})^{24}}\right|^{2}\right)^{\ell}.\notag
\end{align*}
Except for $\frac{d^{\ell}\delta_{2}^{2\ell-1}(E_{2})(\tau_{\mathcal{A}})}{\eta(\tau_{\mathcal{A}})^{8\ell}}$, all the terms are known to be algebraic integers by Lemmas~\ref{etaAetaO}. Invoking~\eqref{deltaeee}, we have
        $$
        \frac{d^{\ell}\delta_{2}^{2\ell-1}(E_{2})}{\eta(\tau)^{8\ell}}=\sum_{\substack{i,j,k\geq 0\\ 2i+4j+6k=4\ell}}c_{i,j,k}d^{\ell}\left(\frac{E_{2}(\tau)^{i}}{\eta(\tau)^{4i}}\right)\left(\frac{E_{4}(\tau)^{j}}{\eta(\tau)^{8j}}\right)\left(\frac{E_{6}(\tau)^{k}}{\eta(\tau)^{12k}}\right)
        $$
        with $c_{i,j,k}\in\mathbb{Z}[\frac{1}{6}]$.  At $\tau=\tau_{\mathcal{A}}=\frac{-b+\sqrt{-d}}{2a}$, by Lemma~\ref{E4E6}, both
         $$
    \frac{E_{4}(\tau_{\mathcal{A}})}{\eta(\tau_{\mathcal{A}})^{8}}\quad\mbox{and}\quad \frac{E_{6}(\tau_{\mathcal{A}})}{\eta(\tau_{\mathcal{A}})^{12}}
    $$
        are algebraic integers. Moreover, Lemma~\ref{fE2} yields that
        $$
        \frac{E_{2}(\tau_{\mathcal{A}})}{\eta(\tau_{\mathcal{A}})^{4}}=\frac{-b+\sqrt{-d}}{2\sqrt{-d}}\cdot\frac{f_{\mathcal{A}}(\tau_{\mathcal{A}})}{\eta(\tau_{\mathcal{A}})^{4}},  
        $$
        so since $2i+4j+6k=4\ell$ with $i,j,k\geq0$, then $i\leq 2\ell$, i.e., $\ell\geq \frac{i}{2}$, and thus, 
        $$
        d^{\ell}\frac{E_{2}(\tau_{\mathcal{A}})^{i}}{\eta(\tau_{\mathcal{A}})^{4i}}= (-1)^{-\frac{i}{2}} d^{\ell-\frac{i}{2}}\left(\frac{-b+\sqrt{-d}}{2}\right)^{i}\left(\frac{f_{\mathcal{A}}(\tau_{\mathcal{A}})}{\eta(\tau_{\mathcal{A}})^{4}}\right)^{i}
        $$
        is $\lambda$-integral for a prime ideal $\lambda$ over $p$ by Proposition~\ref{fAetaA}. Combining all the information above justifies the desired conclusion in the theorem.
\end{proof}


\section{Proof of Theorem~\ref{mainthm2}}
\label{proofthm2}

The technical lemma below shows that  any real algebraic number satisfying the membership indicated in Propositions~\ref{inKab} and~\ref{cuberational} is rational.

\begin{lemma}
    \label{tech2}
    Let $K\ne \mathbb{Q}(\zeta_{3})$ be an imaginary quadratic field and let $x$ be a real number. If $x\in K_{ab}$ and $x^{3}\in K$, then $x\in\mathbb{Q}$.
\end{lemma}

\begin{proof}
    If $x\not\in \mathbb{Q}$, then $[K(x):K]=3$, since $x^{3}\in K$ and $x$ is real. As $x\in K_{ab}$, the subfield $K(x)$ must be Galois over~$K$, therefore, $\zeta_{3}x\in K(x)$, whence $\zeta_{3}\in K(x)$. However, since $[K(x):K]=3$, one must have $\zeta_{3}\in K$, i.e., $K=\mathbb{Q}(\zeta_{3})$, a contradiction to the assumption on~$K$. 
\end{proof}

\begin{proof}[Proof of Theorem~\ref{mainthm2}]
This follows from combining Propositions~\ref{inKab} and~\ref{cuberational}, and Lemma~\ref{tech2}.

\end{proof}

\end{document}